\documentclass[12pt,twoside]{amsart}
\usepackage{amsmath,mathrsfs,dsfont,amsthm}
\usepackage{amscd,amssymb,epsfig,yhmath}
\usepackage{wasysym}
\usepackage{pifont}
\usepackage{skak}
\usepackage{hyperref}

\newcommand{\abs}[1]{\vert#1\vert}
\def\norm#1{\left\Vert#1\right\Vert}
\def\R {{\mathds R}}

\def\N{{\mathds N}}
\def\Id{{\mathbb{Id}}}
\def\e{{\varepsilon}}

\def\H{{\mathcal H}}

\def\lin{{\mathrm{lin}\,}}

\hypersetup{
  colorlinks=true,       
  linkcolor=magenta,          
    filecolor=magenta,      
    urlcolor=cyan           
}

\newtheorem{theorem}{Theorem}[section]
\newtheorem{corollary}[theorem]{Corollary}
\newtheorem{lemma}[theorem]{Lemma}
\newtheorem{proposition}[theorem]{Proposition}
\theoremstyle{definition}

\theoremstyle{remark}
\newtheorem{remark}[theorem]{Remark}

\newtheorem{example}[theorem]{Example}
\numberwithin{equation}{section}

\begin{document}

\begin{center}
\renewcommand{\thefootnote}{\fnsymbol{footnote}}
{\large\bf On finite-dimensional encoding/decoding theorems for neural operators}
\\[.6cm]

{\sc Vin\'\i cius Luz Oliveira$^{ a}$ and Vladimir G. Pestov$^{ b,c}$}
\\[.5cm]
$^{ a}$Departamento de F\'\i sica, Universidade Federal de Santa Catarina, Trindade, Florian\'opolis, SC, 88035-972, Brazil
\\[1mm]
viniluzoliveira@gmail.com
\\[1mm]
$^{ b}$Departamento de Matem\'atica, Universidade Federal de Santa Catarina, Trindade, Florian\'opolis, SC, 88040-900, Brazil
\\[.5mm]
$^{ c}$Department of Mathematics and Statistics, University of Ottawa,
Ottawa, ON \\
K1N 6N5 Canada \footnote[5]{Emeritus Professor.}
\\[1mm]
vpest283@uottawa.ca
\\[.5cm]
\end{center}

{\footnotesize\em \clock\hskip .3cm  Version as of 11:41 BRT, Jan. 20, 2026}
\vskip .2cm

\hrule
\vskip .4cm

{\footnotesize\bf Abstract}

\begin{quote}
{\footnotesize
  \noindent
  Recently, versions of neural networks with infinite-dimensional affine operators inside the computational units (``neural operator'' networks) have been applied to learn solutions to differential equations. To enable practical computations, one employs finite-dimensional encoding/decoding theorems of the following kind: every continuous mapping $f$ between function spaces $E$ and $F$ is approximated in the topology of uniform convergence on compacta by continuous mappings factoring through two finite dimensional Banach spaces.
Such a result is known (Kovachki et al., 2023) for $E,F$ being Banach spaces having the approximation property. We point out that the result needs no assumptions on $E,F$ whatsoever and remains true not only for all normed spaces, but for arbitrary locally convex spaces as well.
At the same time, an analogous result for $C^k$-smooth mappings and the $C^k$ compact open topology, $k\geq 1$, holds if and only if the space $E$ has the approximation property.
This analysis may be useful already because non-normable locally convex function spaces are common in the theory of differential equations, the main field of applications for the emerging theory.
}
\end{quote}
\vskip .4cm
\hrule

\vskip 1cm

\section{Introduction}
Every computational unit of a traditional neural network is, from a mathematical viewpoint, a function obtained by composing an affine functional on a finite-dimensional vector space, $x\mapsto\sum_{i=1}^nw_ix_i+\theta$, with a non-linear activation function, $\sigma$. If there are $m$ units in a layer, we can view them as coordinate components of a single map from $\R^n$ to $\R^m$, obtained by composing an affine operator from $\R^n$ to $\R^m$ (a composition of a linear operator with a shift by a vector of thresholds), and a non-linear function from $\R^m$ to itself.

Now one can think of generalizing this by allowing the affine operator and the activation mapping to act between {\em infinite-dimensional topological vector spaces} of inputs and outputs, in which case the linear operator will have to be continuous and perhaps subject to more stringent restrictions, such as being compact. Recently it was observed \cite{LJK} that such idealized neural networks with infinite-dimensional computational units can be useful for modelling the learning process of solutions to differential equations. In those models, the infinite-dimensional spaces of inputs and outputs are typically function spaces of coefficients and of solutions to a differential equation, and the the affine operators are parameters to be learned.
Without attempting to describe the architecture and functioning of those ``neural operator'' networks in our note, nor to survey the existing work in the area (which has exploded over just a few years), we refer the reader instead to the recent survey \cite{KLS}, including a literature review with a great many references.

Clearly, such neural operator networks are a mere idealization tool, and the practical computations can only handle mappings between spaces of a very high, but still finite, dimension. For this reason, any mapping computed by a ``neural operator'' network needs to be reducible to a mapping between finite-dimensional vector spaces. This is achieved with the help of the so-called finite-dimensional {\em encoding/decoding} (or: {\em latent structure}) theorems, see Sect. 2.4 and elsewhere in \cite{KLS} for a discussion.

Given two function spaces $E$ and $F$ and a continuous mapping $f\colon E\to F$, such a theorem seeks to approximate $f$ with continuous mappings that factor through two finite-dimensional spaces.
The approximation is understood in the sense of uniform convergence on compacta.
\begin{figure}[ht]
\begin{center}
\scalebox{1.5}{\includegraphics{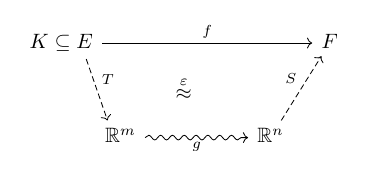}}
\caption
{Decoding/encoding}
\end{center}
\end{figure}

Thus, given a compact subset $K\subseteq E$ and an $\e>0$, we seek 
two finite-dimensional vector spaces, $E_1\cong \R^m$ and $F_1\cong\R^n$, two continuous linear operators, $T\colon E\to E_1$ and $S\colon F_1\to F$, and a continuous function $g\colon F_1\to E_1$ having the property:
\[\max_{x\in K}\norm{f(x)-S\circ g\circ T(x)}_F<\e.\]

One case of obvious importance is where $E$ and $F$ are both Hilbert spaces (for nstance, Sobolev spaces). For $E,F$ Hilbert, the possibility of finite-dimensional encoding/decoding was proved in \cite{BHKS}. In (\cite{KLLABSA}, Lemma 22), the result was extended to the case where $E$ and $F$ are Banach spaces having the {\em approximation property} in the sense of classical Banach space theory.

The first aim of our note is to remove any restrictions on $E$ and $F$ whatsoever and to state the theorem in about the most general form possible.

\begin{theorem}[Finite encoding/decoding for neural operators, continuous case]
  Let $f\colon E\to F$ be a continuous mapping between two locally convex spaces. It is approximated in the topology of uniform convergence on compacta by mappings of the form $S\circ g\circ T$, where $S\colon F_1\to F$ and $T\colon R\to E_1$ are bounded linear operators, $E_1,F_1$ are finite-dimensional, and $g\colon E_1\to F_1$ is continuous.
  \label{th:main}
\end{theorem}

In particular, the theorem holds for arbitrary normed space $E$ and $F$. At the same time, many function spaces of central importance in theory of differential equations are non-normable locally convex spaces: various concrete Fr\'echet spaces, nuclear spaces, $(LF)$-spaces, and so on, see, for example, the monograph \cite{treves}.

Under certain regularity assumptions, the solutions to differential equations depend on the initial conditions in a smooth way. This, in our view, justifies extending the encoding/decoding result to the case of $C^k$ smooth mappings. While in the finite-dimensional case the notion of $C^k$ smoothness, as well as the topology on the space of such mappings -- the uniform convergence on compacta with $k$ first derivatives -- are uniquely defined, for the mappings between infinite-dimensional spaces there is a great variety of divergent concepts, see the monograph \cite{keller}. Following \cite{neeb}, we adopt the definition of $C^k$ mappings due to Michal \cite{michal} e Bastiani \cite{bastiani}, as well as the topology of uniform convergence on compacta together with $k$ derivatives. Here the situation is different from the continuous case. Recall that a locally convex space $E$ is said to have the {\em approximation property} if the identity mapping $\Id\colon E\to E$ is approximated, in the topology of uniform convergence on compacta, with continuous linear operators $\pi\colon E\to E$ of finite rank. 

\begin{theorem}[Finite encoding/decoding for neural operators, smooth case]
  For a complete locally convex space $E$, the following are equivalent.
  \begin{enumerate}
  \item
    For every locally convex space $F$, every $C^k$-smooth mapping $f\colon E\to F$, $k\geq 1$, is approximated in the topology of uniform convergence on compacta with the first $k$ derivaties by mappings of the form $S\circ g\circ T$, where $S\colon F_1\to F$ and $T\colon E\to E_1$ are bounded linear operators, $E_1,F_1$ are finite-dimensional, and $g\colon E_1\to F_1$ is a $C^k$ smooth mapping.
    \item $E$ has the approximation property.
      \end{enumerate}
  \label{th:main2}
\end{theorem}

In Section \ref{s:background}, we collect various results and concepts from topology, functional analysis, and smooth calculus in infinite dimensions that will be used in the proofs. Section \ref{s:proofs} contains the proof of Theorem \ref{th:main}. First, we give a separate proof for the particular case where $E$ and $F$ are arbitrary normed spaces, because it is technically simpler, yet gives the gist of the proof in the general case. Some readers may just want to stop there; others may wish to go directly to the proof of the general locally convex situation, which is independent. Theorem \ref{th:main2} is proved in Section \ref{s:smooth}.

Not being sure if Theorem \ref{th:main} is already known, we have contacted two among the top experts in modern linear and non-linear functional analysis. Gilles Godefroy, who believes Theorem \ref{th:main} to be a folklore result due to its generality, nevertheless
referred us for a definitive opinion to William Johnson, who informed us he has never seen the result before (and suggested a simpler proof of a part of it). We are grateful to both mathematicians.

In truth, most of the techniques used in the proofs are completely standard, except maybe the construction of the finite-dimensional space $E_1$ in the continuous case which avoids the use of approximation property. Upon reflection, the authors have decided that publishing the results is justified and could benefit the fast growing neural operator community at this point in time.

\section{Background from functional analysis and topology\label{s:background}}

\subsection{Some notions and results of general topology}

All the topological spaces appearing in our note are Hausdorff, that is, given two points $x,y$, $x\neq y$, there exist disjoint neighbourhoods $V_x,V_y$ of $x,y$ respectively.

Recall that a topological space $X$ is {\em separable} if it contains a countable subset which is dense, that is, whose closure is all of $X$.

\begin{proposition}[Th. 4.1.18 in \cite{engelking}]
  A compact metrizable space is separable.
  \label{p:sep}
 \end{proposition}
(Just take the union of centres of finite families of open balls of radii $1/k$ covering $X$, for all $k\geq 1$.)

A continuous mapping $f\colon X\to Y$ is called {\em closed} if the image $f(F)$ of every closed set $F\subseteq X$ is closed in $Y$. In particular,

\begin{proposition}[Th. 3.1.12 in \cite{engelking}]
  Every continuous map from a compact space $X$ to a Hausdorff space $Y$ is closed.
  \label{p:closed}
\end{proposition}

Indeed, if $F\subseteq X$ is closed, it is compact, hence its image under the continuous map $f$ is compact in $Y$, and since $Y$ is Hausdorff, it is closed.

A surjective map $f\colon X\to Y$ between two topological spaces is a {\em quotient map} if a subset $F\subseteq Y$ is closed in $Y$ if and only if $f^{-1}(Y)$ is closed in $X$ (see \cite{engelking}, p. 91). In particular, a quotient map is continuous. 

\begin{proposition}[Corol. 2.4.8, \cite{engelking}]
  Every surjective closed map is a quotient map.
  \label{p:quotient}
\end{proposition}

Indeed, if $F\subseteq Y$ is such that $f^{-1}(F)$ is closed in $X$, then, by surjectivity, $F=f(f^{-1}(F))$, so $F$ is closed in $Y$.

A bijective quotient map is clearly a homeomorphism (has a continuous inverse), see Corol. 2.4.7 in \cite{engelking}. So in particular we have the following well-known and useful observation.

\begin{proposition}[Th. 3.1.13, \cite{engelking}]
  A continuous injective map $f$ from a compact space $X$ onto a Hausdorff topological space $Y$ is a homeomorphism, that is, has a continuous inverse.
  \label{p:homeo}
\end{proposition}

\begin{proposition}[Prop. 2.4.2, \cite{engelking}]
  Let $f\colon X\to Y$ be a quotient map and $g\colon Y\to Z$ be such that $gf\colon X\to Z$ is continuous. Then $g$ is continuous.
  \label{p:3maps}
\end{proposition}

\begin{proof}
  If $F\subseteq Z$ is closed, then $f^{-1}(g^{-1}(F)) = (gf)^{-1}(F)$ is closed in $X$. Since $f$ is a quotient map, $g^{-1}(F)$ is closed in $Y$.
\end{proof}

A ({\em continuous}) {\em partition of unity} on a topological space $X$ is a family of continuous maps $h_i\colon X\to [0,1]$, indexed by some index set $I$, having the property that $\sum_{i\in I}h_i(x)=1$ for each $x\in X$. 

An cover $\gamma$ of a topological space $X$ is called {\em locally finite} if every point $x\in X$ has a neighbourhood $V$ meeting only finitely many elements of $\gamma$. In particular, every finite cover is locally finite, and this is the only case of interest to us.

We will only apply the following result in the case where $X$ is compact, so will not discuss the more general concept of paracompactness, for which notion we refer to Chapter V in \cite{engelking}. Suffices only to mention that each compact space is paracompact ({\em ibid.}, Th. 5.1.1 on p. 300).

\begin{theorem}
  Let $\gamma$ be a locally finite open cover of a paracompact Hausdorff topological space. There exists a partition of unity, $h_V$, $V\in\gamma$, indexed with elements of $\gamma$ and such that for every $V\in \gamma$, the function $h_V$ vanishes outside of $V$.
  \label{th:unity}
\end{theorem}

See Theorem 5.1.9 in \cite{engelking}. Even if the result is not stated in exactly this form, the above affirmation follows from the proof of the implication (i)$\Rightarrow$(ii), on top of the page 302. 

\subsection{Normed and Hilbert spaces}

All vector spaces in this work will be over the field of real numbers.
Recall that $C[0,1]$ denotes the Banach space consisting of all continuous real-valued functions on the closed unit interval, equipped with the uniform (supremum) norm.

\begin{theorem}[Banach--Mazur theorem, \cite{banach}, p. 185]
Every separable normed space is isometrically isomorphic to a normed subspace of $C[0,1]$. 
\end{theorem}

A more modern presentation of this classical result can be found, for example, in \cite{KG}, Th. 24, or else \cite{HMVZ}, Ch. 2.

Note that the identity linear operator $C[0,1]\to L^2(0,1)$ has norm $1$. Given a separable normed space $E$, we can compose the Banach--Mazur isometric linear embedding $i\colon E\hookrightarrow C[0,1]$ with the identity operator $C[0,1]\to L^2(0,1)$, to obtain:

\begin{corollary}
  Every separable normed space admits an injective linear operator of operator norm $1$ into a separable Hilbert space.
  \label{c:iell2}
\end{corollary}

The following result is very well-known, see e.g. p. 108 in \cite{schaefer}.

\begin{theorem}[Approximation Property of the Hilbert space]
  Let $\pi_n$ be a sequence of orthogonal projections in a separable Hilbert space $\H$ whose ranges form an increasing chain with dense union in $\H$. For every compact $K\subseteq \H$ and each $\e>0$, there is $N$ so that for all $n\geq N$ and each $x\in K$, $\norm{\pi_n(x)-x}<\e$.
  \label{th:approx}
\end{theorem}

Select a finite subset $A\subseteq\H$ forming an $\e/2$-net for $K$, that is, for $x\in K$ there is $a\in A$ with $\norm{x-a}<\e/2$. Choose $N$ so that for all $n\geq N$ and for each $a\in A$ the distance from $a$ to the range of $\pi_n$ is less than $\e/2$.
Since $\pi_n(x)$ is the closest point to $x$ in the range of $\pi_n$, for all $x\in K$ we have $\norm{x-\pi_n(x)}<\e$.

\subsection{Locally convex spaces}
Some standard references to the theory of locally convex spaces are \cite{schaefer,jarchow}, see also \cite{treves2}. Here is a brief recap of the few basic notions and results we will need.

A vector space equipped with a Hausdorff topology is called a {\em topological vector space} if the addition and scalar multiplication are jointly continuous. A topological vector space is called a {\em locally convex space} if there is a basic system of neighbourhoods of zero consisting of convex sets.

For example, every normed space is locally convex. A finite-dimensional vector space supports a unique topology making it a topological vector space, which is therefore locally convex.

The Hahn--Banach extension theorem is one of the most basic results of the theory, and we will use a rather limited version of it.

\begin{theorem}[Hahn--Banach theorem; Th. 3.2 in \cite{schaefer}; Sect. 7.2 in \cite{jarchow}; Th. 18.1 in \cite{treves2}]
  Let $F$ be a topological vector subspace of a locally convex space $E$. Every continuous linear functional $f\colon F\to\R$ admits an extension to a continuous linear functional $\tilde f\colon E\to\R$.
  \label{th:hb}
\end{theorem}

\begin{corollary}
  Let $F$ be a topological vector subspace of a locally convex space $E$, and let $T\colon F\to \R^m$ be a continuous linear operator. Then $T$ extends to a continuous linear operator from $E$ to $\R^m$.
  \label{c:hb}
\end{corollary}

\begin{proof}
  Enough to extend each one of the $m$ coordinate components of $T$ separately by the Hahn--Banach theorem, and put them together.
  \end{proof}

A {\em prenorm} (or: {\em semi-norm}) on a vector space $E$ is a map $p\colon E\to \R$ satisfying the axioms of the norm except that $p(x)=0$ does not necessarily imply $x=0$. Thus, we have:
\begin{itemize}
\item $p(\lambda x)=\abs\lambda p(x)$,
  \item $p(x+y)\leq p(x)+p(y)$.
\end{itemize}
(The non-negativity of $p(x)$ and the fact that $p(0)=0$ follow.)

The open and closed balls with regard to a prenorm $p$ are defined just like in the case of a norm:
\[B^p_\e(x)=\{y\in E\colon p(x-y)<\e\},~~\bar B^p_\e(x)=\{y\in E\colon p(x-y)\leq\e\}.\]
We will add the prenorm symbol in such cases.

A basic result of the theory of locally convex spaces states that for every neighbourhood of zero, $W$, there is a continuous prenorm $p$ on $E$ with $B^p_1(0)\subseteq W$ (\cite{jarchow}, Sect. 7.5; \cite{schaefer}, p. 39; \cite{treves2}, Prop. 7.5).

If $\pi\colon E\to F$ is a linear operator and $F$ is a normed space, we can define a prenorm on $E$ by $p(x)=\norm{\pi(x)}_F$. If $E$ is a locally convex space and $\pi$ is continuous, then so is the prenorm $p$. Conversely,
given a prenorm, $p$, on a vector space $E$, the set $N_p=\{x\in E\colon p(x)=0\}$ forms a linear subspace, and the quotient space $E/N_p$ (often denoted simply $E/p$) is equipped with the norm, consistently defined by $\norm{x+N_p}=p(x)$.

A family $P$ of prenorms on a vector space $E$ {\em separates points} if for each $x\neq 0$ there is $p\in P$ with $p(x)\neq 0$. Every such family $P$ determines a Hausdorff locally convex topology: the weakest topology in which all prenorms from the family are continuous, or, equivalently, the weakest topology in which every ball $B^p_\e(x)$, $\e>0$, $x\in E$, $p\in P$ is open. (See Prop. 7.6 in \cite{treves2}.)

Examples of locally convex spaces that are not normable are given by direct products. If $(E_\alpha)$ is a family of normed spaces, the direct product $\prod_{\alpha}E_{\alpha}$ is a vector space, and we equip it with the product topology: the topology generated by the prenorms of the form $\norm{\pi_\alpha(x)}_{E_\alpha}$ for all $\alpha$. If the family is infinite and each $E_\alpha$ is nontrivial, the product space is (easily seen to be) non-normable.

Again, let $P$ be a family of continuous prenorms on a locally convex space $E$. There is a canonical continuous linear operator from $E$ to the product $\prod_{p\in P}E/p$, given by $\varpi(x)_{\alpha} = \pi_{\alpha}(x)$, where $\pi_{\alpha}\colon E\to E/p$ denotes the quotient mapping. If the family $P$ generates the topology of $E$, then the operator $\varpi$ is a topological embedding.

If the topology of a locally convex space $E$ is generated by a family of prenorms $P$ with the property that each $E/p$ is a pre-Hilbert space, then $E$ is called a {\em multi-Hilbert space.}

A locally convex space $E$ is {\em complete} if it has the following property. Let $\mathscr F$ be a non-empty family of non-empty subsets of $E$ closed under finite intersections and such that for every neighbourhood $V$ of zero there is $F\in\mathscr F$ with $x-y\in V$ for each $x,y\in F$. Then there is a (necessarily unique) element $x\in E$ whose every neighbourhood contains as a subset an element of $\mathscr F$. See Ch. 3 in \cite{jarchow} or Sect. 1.4 in \cite{schaefer}.
For example, a normed space is complete in this sense if and only if it is a Banach space. All the locally convex spaces that would play any role in the theory of differential equations are invariably complete.
The product of a family of complete locally convex spaces is complete (Prop. 3.2.6 in \cite{jarchow}).

A complete locally convex space $E$ has the {\em approximation property} if the identity map $\Id_E\colon E\to E$ is approximated, uniformly on compact sets, with linear operators $T\colon E\to E$ of finite rank. That is, given a compact $K\subseteq E$ and an open $U$ containing zero, there is a continuous linear operator $T\colon E\to E$ with $T(E)$ finite dimensional and such that for all $x\in K$, $x-Tx\in U$.

There are slight variations of this concept, but they all coincide between themselves and with the above in the most important case where $E$ is {\em complete}. See Remark on p. 399 in \cite{jarchow}.

\begin{proposition}[Corol. 18.2.4 in \cite{jarchow}]
  The product of a family of locally convex spaces with approximation property has the approximation property.
\end{proposition}

We will only need a very particular case:

\begin{corollary}[See Corol. 9.2.1 in \cite{schaefer}]
  The product of a family of Hilbert spaces with the product topology has the approximation property.
  \label{c:multihilbertAP}
  \end{corollary}

\subsection{Uniform continuity}

We will be using in an essential way the uniform continuity of functions defined on compact spaces, on locally convex spaces, and on metric spaces. The unifying viewpoint is provided by the notion of a {\em uniform space}. A ({\em separated}) {\em uniform space} is a set, $X$, equipped with a {\em uniform structure,} which is a collection, $\mathcal U$, of subsets of the square $X\times X$ all containing the diagonal $\Delta=\{(x,x)\colon x\in X\}$ and satisfying the properties:
\begin{itemize}
\item If $U\in \mathcal U$ and $U\subseteq V\subseteq X\times X$, then $V\in\mathcal U$.
\item If $U\in {\mathcal U}$, then $U^{-1}=\{(y,x)\colon (x,y)\in U\}\in\mathcal U$.
\item If $U\in {\mathcal U}$, there is $V\in\mathcal U$ such that
  \[V^2 = \{(x,z)\colon \exists y\in X,~(x,y)\in V,~(y,z)\in V\}\subseteq U.\]
\item If $x,y\in X$ and $x\neq y$, then for some $U\in\mathscr U$, $(x,y)\notin U$.
\end{itemize}

The sets $V\in \mathcal U$ are called {\em entourages of the diagonal}. For two such entourages, $V$ and $U$, their composition is defined as
\[V\circ U=\{(x,z)\colon \exists y,~~(x,y)\in V,~(y,z)\in U\}.\]
A particular case is the square of an entourage, $V^2=V\circ V$, that appeared in one of the axioms. 

Three important sources of examples of uniform spaces are the following.

(1) If $(X,d)$ is a metric space, then the uniform structure on $X$ generated by the metric $d$ consists of all sets $V\subseteq X\times X$ for which there is $\e>0$ with
\[\{(x,y)\in X\times X\colon d(x,y)<\e\}\subseteq V.\]
This example helps to see that the axioms of a uniform space mirror the axioms of a metric.

(2) If $E$ is a locally convex space, then the additive uniform structure on $E$ consists of all $U\subseteq E\times E$ such that for some neighbourhood of zero, $W$, in $E$ we have $(x,y)\in U$ as soon as $x-y\in W$.

If $E$ is a normed space, this is the same uniform structure as generated by the metric associated with the norm: $d(x,y)=\norm{x-y}_E$.

(3) And if $K$ is a (not necessarily metrizable) compact space, the canonical uniform structure on $K$ is given by all neighbourhoods of the diagonal, that is, all sets $V\subseteq K\times K$ whose interior in the product topology contains the diagonal, $\Delta$. 

Every uniform space supports a canonical topology: the basic neighbourhoods of a point $x\in X$ are formed by the sets $U[x]=\{y\in X\colon (x,y)\in U\}$, where $U\in\mathcal U$. For a metric space with the metric uniform structure this is the same topology as generated by the metric. For a locally convex space, the additive uniform structure generates the original topology. For a compact space, the unique uniformity generates the topology of the space.

A function $f\colon (X,{\mathcal U}_X)\to (Y,{\mathcal U}_Y)$ between two uniform spaces is {\em uniformly continuous} if for every $U\in {\mathcal U}_Y$ there is $V\in {\mathcal U}_X$ with $(f(x),f(y))\in U$ for all $(x,y)\in V$. It is easy to see that in the case of metric spaces, this definition is equivalent to the classical one:
\[\forall \e>0,~\exists\delta>0~,\forall x,y\in X,~~d_Y(f(x),f(y))<\e.\]

In case of a function $f$ between two locally convex spaces, $E$ and $F$, the definition becomes: for every neighbourhood, $V$, of zero in $F$, there is a neighbourhood of zero, $W$, in $E$, such that, if $x-y\in W$, then $f(x)-f(y)\in V$.

And it is easily verified, directly from the definitions, that every continuous mapping $f$ from a compact space to a uniform space is automatically uniformly continuous: it is because the mapping $f\times f$, sending $(x,y)$ to $(f(x),f(y))$, is continuous. (In this abstract setting the proof becomes instantanous, more transparent than in the textbook case of a uniformly continuous real-valued function on a metrizable compact space.)

As a consequence of this fact, a Hausdorff compact space supports a unique uniform structure generating the topology. For this reason, if $K$ is a metrizable compact space, any compatible metric on $K$ will generate the canonical uniform structure, and the notions of uniform continuity for functions defined on $K$ are all equivalent, whether we are using a metric generating the topology of $K$, or the additive uniform structure of a locally convex space containing $K$ as a topological subspace.

For all of this, we refer the reader to Chapter 8 in \cite{engelking}, Sections 8.1--8.3, in particular Th. 8.3.13.

\begin{lemma}
  Let $f\colon X\to Y$ be a mapping between two uniform spaces that is continuous with regard to the corresponding topologies. Let $K\subseteq X$ be a compact subset. For every entourage $V\in {\mathcal U}_Y$ there is an entourage $U\in {\mathcal U}_X$ with the property: if $x\in K$, $y\in X$, and $(x,y)\in U$, then $(f(x),f(y))\in V$. 
\end{lemma}

\begin{proof}
  For every $x\in K$, use the continuity of $f$ at $x$ to select a symmetric entourage $U_x\in {\mathcal U}_X$ so that for each $y\in X$ satisfying $(x,y)\in U_x^2$, we have $(f(x),f(y))\in V$. Using the compactness of $K$, select a finite set $(x_i)_{i=1}^n$ with open sets $U_{x_i}[x_i]$ covering $K$. Set $U=\cap_{i=1}^n U_{x_i}$. If now $x\in K$ and $y\in X$ satisfy $(x,y)\in U$, then for some $i=1,2,\ldots,n$ we have $(x_i,x)\in U_{x_i}$, therefore $(x_i,y)\in U_{x_i}\circ U\subseteq U_{x_i}^2$, and the statement follows by the choice of $U_{x_i}$.
\end{proof}

In particular, we have as the corollary:

\begin{lemma}
  Let $f\colon E\to F$ be a continuous mapping between two locally convex spaces, let $K$ be a compact subset of $E$, and let $V$ be a neighbourhood of zero in $F$. There is a neighbourhood of zero $U$ in $E$ such that if $x\in K$, $y\in E$, and $x-y\in U$, then $f(x)-f(y)\in V$.
  \label{l:uniformnbhd}
  \end{lemma}

We will also need the Kat\v etov theorem.

\begin{theorem}[Kat\v etov Extension Theorem, \cite{katetov}, Th. 3]
  Every bounded uniformly continuous real-valued function defined on a subspace of a uniform space extends to a bounded uniformly continuous function on the entire space.
  \label{th:katetov}
\end{theorem}

If a function $f$ on a uniform space is taking values in a finite-dimensional vector space $F$ with its unique topology (hence, the uniform structure), then, by identifying $F$ with $\R^{\dim F}$ and applying Katetov Extension Theorem separately to each coordinate component, we obtain:

\begin{corollary}
  Every bounded uniformly continuous function with values in a finite-dimensional vector space defined on a subspace of a uniform space extends to a bounded uniformly continuous function on the entire space.
  \label{c:katetov}
\end{corollary}

Let us remark in passing that the uniform spaces provide a natural general framework for the notion of completeness as well. A uniform space $X$ is complete if it has the following property. Let  $\mathscr F$ be a non-empty family of non-empty subsets of $X$, closed under finite intersections and such that for every $U\in{\mathcal U}_X$ there is $F\in{\mathscr F}$ with $F\times F\subseteq U$ (the Cauchy property of $\mathscr F$). Then there is a (unique) $x\in X$ (the limit of $\mathscr F$) whose every neighbourhood contains an element of $\mathscr F$ as a subset. In this way, the concept of a uniform space, introduced by Andr\'e Weil in the monograph \cite{weil}, provides the most general unifying framework allowing to treat both uniform continuity and completeness.

\subsection{Calculus in locally convex spaces\label{s:calculus}}

Let $U\subseteq E$ be an open subset of a locally convex space $E$, and let $f\colon U\to F$ be a function with values in another locally convex space. The function $f$ is said to be {\em differentiable at a point} $x\in U$ if for every $v\in E$ there exists a limit
\[D^1f(x)(v) = \lim_{t\to 0}\frac 1t\left[f(x+tv)-f(x) \right].\]
One says that $f$ is $C^1$ on $U$ if $f$ itself is continuous, the derivative exists at every point $x\in U$, and the resulting function
\[D^1f\colon U\times E\to F\]
is continuous. (This definition corresponds to the concept of Michal--Bastiani $C^1$-function \cite{michal,bastiani}, and to the concept of a $C^1_{\mathbf k}$ function in \cite{keller}.)

Now one can define recursively the $k$-th derivative of $f$, $k\geq 1$: for all $x\in U$ and $v_1,\ldots,v_k\in E$,
\[D^kf(x)(v_1,v_2,\ldots,v_k) = D^1\left(D^{k-1}(x)(v_1,\ldots,v_{k-1}) \right)(v_k).\]
A function $f$ is of class $C^k$ if the $k$-th derivative exists and all the maps
\[D^if\colon U\times E^i\to F,\]
$i=0,1,\ldots,k$, are continuous.

Fixing $x$, we obtain a map $D^kf(x)\colon E^k\to F$. 
For $C^k$-functions, this map is $k$-linear and symmetric; one can assume that $E$ has dimension $k+1$ and (with the help of Hahn--Banach theorem) that $F=\R$, and now it follows from the classical case.

\begin{example}
If $T\colon E\to F$ is a linear operator, then for every $x,v\in E$ we have
\begin{align*}
  D^1T(x)(v) &= \frac{d}{dt}\vert_{t=0}[T(x+tv)-T(x)]\\
&=  \frac{d}{dt}\vert_{t=0}tT(v) \\
&= T(v),
\end{align*}
that is, $D^1T(x)= T$ at every $x$.
\label{ex:linear}
\end{example}

The collection of all $C^k$ functions from an open subset $U$ of $E$ to $F$ forms a linear space.
We denote this space $C^k(U,F)$.
The topology of uniform convergence on compacta together with $k$ first derivatives is given on this space by the following collection of prenorms:
\[p_{K_1,K_2,i}(f) = \max_{x\in K_1,v_1,\ldots,v_i\in K_2}p\left(D^if(x)(v_1,\ldots,v_i)\right),\]
where $i=0,1,\ldots,k$, $K_1$ is a compact subset of $U$, $K_2$ is a compact subset of $E$ and $p$ is a continuous prenorm on $F$. This topology makes $C^k(U,F)$ into a locally convex space. One can equally well describe this topology via basic neighbourhoods of a function $f$: those consist of all $C^k$-smooth functions $g\colon U\to F$ satisfying
\[\forall i=0,1,\ldots,k,~~\forall x\in K_1,\forall v_1,\ldots,v_i\in K_2,~D^if(x)(v_1,\ldots,v_i)-D^ig(x)(v_1,\ldots,v_i)\in V,\]
where $K_1,K_2$ are as above and $V$ is a neighbourhood of zero in $F$.
See \cite{neeb}, Def. I.5.1 on p. 317. The special case where $E$ is finite-dimensional and $F$ is an arbitrary locally convex space appears in Ch. 40 of \cite{treves2}, p. 412. If both $E$ and $F$ have finite dimension, we get the usual $C^k$-topology on the space of $C^k$ functions.

\begin{theorem}[Chain rule; Prop. I.2.3(vi) in \cite{neeb}]
  Lert $E,F,G$ be locally convex spaces.
  The composition of two $C^1$ mappings, $f\colon U\to V$ and $g\colon V\to G$, where $U$ is open in $E$ and $V$ is open in $F$,
  is again of class $C^1$, and for $x\in U$
\[D^1(g\circ f)(x) = D^1g(f(x))\circ D^1f(x).\]
\label{th:chain}
\end{theorem}

\begin{corollary}[{\em ibid.}] The composition of two $C^k$ mappings is of class $C^k$.
\end{corollary}

\begin{lemma}
  Let $f\colon U\to F$ be a $C^k$ mapping from an open subset $U$ of a locally convex space $E_1$ to a locally convex space $F$, and let $T\colon E\to E_1$ be a continuous linear operator. Then for all $x\in T^{-1}(U)$ and $v_1,\ldots,v_k\in E$,
  \[D^k(f\circ T)(x)(v_1,v_2,\ldots,v_k)= D^kf(T(x))(T(v_1),\ldots,T(v_k)).\]
  \label{l:compo}
\end{lemma}

\begin{proof}
  The Chain Rule and example \ref{ex:linear} imply the formula for $k=1$, that is, the derivative of the composition of a $C^1$ function $f$ with a linear operator $T$:
  \begin{equation}
    D^1(f\circ T)(x)(v) = D^1f(T(x))(T(v)).
    \label{eq:k=1}
  \end{equation}
  Now suppose the formula is verified for some $i$, $1\leq i\leq k-1$:
  \[D^i(f\circ T)(x)(v_1,v_2,\ldots,v_i)= D^if(T(x))(T(v_1),\ldots,T(v_i)).\]
  Fix the values $v_1,\ldots,v_i$. Then
  \[D^{i+1}(f\circ T)(x)(v_1,\ldots,v_{i+1})=
  D^1\left(D^i(f\circ T)(x)(v_1,v_2,\ldots,v_i) \right)(v_{i+1})\]
  is the derivative at the point $x$ of the composition of the function \[y\mapsto D^if(y)(T(v_1),\ldots,T(v_i))\]
  with the linear operator $T$ (here we used the induction hypothesis), evaluated at $v_{i+1}$. The formula (\ref{eq:k=1}) gives
  the desired expression, $D^{i+1}f(T(x))(T(v_1),\ldots,T(v_i),T(v_{i+1}))$.
  \end{proof}

\subsection{Tensor products}
Given two vector spaces, $E$ and $F$, their {\em algebraic tensor product}, $E\otimes F$, is a vector space spanned by the elementary tensors, $e\otimes f$, $e\in E$, $f\in F$, subject to the defining relations $(\lambda a+\mu b)\otimes c=\lambda (a\otimes c)+\mu(a\otimes c)$ and $a\otimes(\lambda b+\mu c)=\lambda (a\otimes c)+\mu (b\otimes c)$; the only other relations between elements of $E\otimes F$ are those derived from these two. See sect. 1.6 in \cite{jarchow} or p. 92 in \cite{schaefer}.

The case of interest for us is where $E=\mathscr F$ is a space of real-valued functions on a set $X$, and $F$ is any vector space. Then an elementary tensor $f\otimes v$ in $\mathscr F\otimes F$ can be interpreted as the mapping 
\[f\otimes v\colon X\ni x\mapsto f(x)v\in F.\]
Thus, an arbitrary element $f$ of $\mathscr F\otimes F$ can be written as a function
\[f(x)=\sum_{i=1}^n f_i(x)\otimes v_i=\sum_{i=1}^nf_i(x)v_i,\]
for some $n\in\N$, $f_i\in\mathscr F$, and $v_i\in F$, $i=1,2,\ldots,n$. In particular, $f$ takes its values in a finite-dimensional vector subspace of $F$ spanned by $v_1,v_2,\ldots,v_n$.

If $\mathscr F=C^k(U,\R)$ consists of real functions of class $C^k$ on an open subset $U$ of a locally convex space, then its tensor product $C^k(U,\R)\otimes F$ with a locally convex space $F$ consists of all functions on $U$ of class $C^k$ with values in $F$ having finite-dimensional ranges.

\begin{theorem}[See \cite{treves2}, Prop. 44.2] Let $U$ be an open subset of $\R^m$, and let $F$ be a locally convex space. The tensor product $C^k(U,\R)\otimes F$ is dense in $C^k(U,F)$ equipped with the $C^k$ compact-open topology.
  \label{th:treves}
  \end{theorem}

\section{Encoding/decoding of continuous functions\label{s:proofs}}

In this Section we give a proof of Theorem \ref{th:main}. 
We will first treat separately the particular case where $E,F$ are normed spaces.

\subsection{Case of normed spaces}

\begin{theorem}
  Let $f\colon E\to F$ be a continuous mapping between two normed spaces. Given a compact set $K\subseteq E$ and $\e>0$, there are finite dimensional normed spaces $E_1,F_1$, bounded linear operators $T\colon E\to E_1$, $S\colon F_1\to F$ and a continuous mapping $g\colon E_1\to F_1$ such that for every $x\in K$,
  \[\norm{f(x)-S\circ g\circ T(x)}_F<\e.\]
\end{theorem}

The rest of the subsection is devoted to the proof.
Let a continuous map $f\colon E\to F$, a compact set $K\subseteq E$ and an $\e>0$ be given.

We start with the construction of the subspace $F_1$.
As the restriction of $f$ to $K$ is uniformly continuous, there is $\delta>0$ so that if $x,y\in K$ and $\norm{x-y}_E<\delta$, then $\norm{f(x)-f(y)}_F<\e$. Select a finite cover of $K$ with open $\delta$-balls, $B_\delta(x_i)$, $x_i\in K$, $i=1,2,\ldots,n$. Fix a continuous partition of unity subordinated to this cover, that is, a family of continuous functions $h_i\colon K\to [0,1]$, vanishing outside of $B_\delta(x_i)$, with $\sum_{i=1}^nh_i\equiv 1\vert_K$ (Theorem \ref{th:unity}). Define for each $x\in K$
\[\tilde f(x) = \sum_{i=1}^n h_i(x)f(x_i).\]
The function $\tilde f$, defined on $K$, is continuous and takes values in the finite-dimensional subspace $F_1$ of $F$ spanned by the vectors $f(x_1)$, $f(x_2)$, $\ldots$, $f(x_n)$. Also, for all $x\in K$ we have
\begin{align}
  \norm{f(x)-\tilde f(x)}_F & = \norm{\sum_{i=1}^n h_i(x)f(x)-\sum_{i=1}^n h_i(x)f(x_i)}_F\nonumber \\
  & = \norm{\sum_{x\in B_\delta(x_i)} h_i(x)\left[f(x)-f(x_i)\right]}_F\nonumber \\
  &\leq \sum_{x\in B_\delta(x_i)} h_i(x) \e \nonumber \\
  &\leq \e.  \label{eq:tildef}
\end{align}

\begin{remark}
  In our original proof, the construction of the space $F_1$ and the function $\tilde f$ were based on the Michael Selection Theorem \cite{michael}. The argument with partitions of unity, indicated to us by William Johnson, is technically simpler.
  \end{remark}

Now we take care of $E$. Since $K$ is a metrizable compact space, it is separable (Prop. \ref{p:sep}), hence the same is true of the closed linear span of $K$ in $E$. Denote this span $\tilde E$. 
Fix a linear injection $i\colon \tilde E\to\ell^2$ having operator norm $1$ (Corol. \ref{c:iell2}). The restriction $i\vert_K$ is a continuous injection, so it is a homeomorphism on its image (Prop. \ref{p:homeo}). The function $\tilde f\circ (i\vert_K)^{-1}$ is therefore a continuous map from the compact subset $i(K)\subseteq\ell^2$ to $F_1$, hence uniformly continuous.
An application of the corollary \ref{c:katetov} to the Kat\v etov Extension Theorem  allows to extend this function
to a uniformly continuous function $\breve f\colon\ell^2\to F_1$, with the property
\begin{equation}
  \breve f\circ i\vert_K = \tilde f\vert_K.
  \label{eq:brevetilde}
  \end{equation}
Now, select $\delta>0$ with the property that for all $x,y\in\ell^2$,
\[\norm{x-y}_2<\delta\Rightarrow\norm{\breve f(x)-\breve f(y)}_{F_1}<\e.\]
Choose $n$ so large that the orthogonal projection $\pi_n$ of $\ell^2$ onto the subspace $\ell^2(n)$ spanned by the first $n$ basic vectors has the property: for all $x\in i(K)$, $\norm{x-\pi_n(x)}_2<\delta$ (Th. \ref{th:approx}). Let us write it like this:
\begin{equation}
  \forall x\in K,~\norm{i(x)-\pi_ni(x)}_2<\delta.\label{eq:pin}
\end{equation}
The composition $\pi_n\circ i\colon \tilde E\to\ell^2(n)$ is a bounded linear operator of norm $1$ with values in a finite dimensional space, and by Corollary \ref{c:hb} of the Hahn--Banach theorem, it extends
to a bounded linear operator from all of $E$ to $\ell^2(n)$, which we denote $T$. 

Now, set:
$E_1=\ell^2(n)$, $g=\breve f\vert_{\ell^2(n)}$, and denote $S$ the canonical embedding of the normed subspace $F_1$ into $F$.

If now $x\in K$, we have, using the notation $a\overset\e\approx b$ to mean $\norm{a-b}< \e$:
\begin{align*}
\mbox{\tiny eq. (\ref{eq:tildef})}~~  f(x) &  \overset \e\approx \tilde f(x) \\
\mbox{\tiny eq. (\ref{eq:brevetilde})} &  = \breve f\circ i(x) \\
  \mbox{\tiny eq. (\ref{eq:pin})}  &  \overset \e\approx \breve f\circ\pi_n\circ i(x) \\
  &= S\circ g\circ T(x).
\end{align*}
This establishes the result.

\subsection{Case of general locally convex spaces}
Now we pass to the proof of the main theorem \ref{th:main}, adapting the same argument as above.

Let $E$ and $F$ be two locally convex spaces, and let $f\colon E\to F$ be a continuous function between them. We are given a compact set $K\subseteq E$ and a neighbourhood of zero, $W$, in $F$, and are asked to produce finite dimensional linear spaces $E_1,F_1$, continuous linear operators $T\colon E\to E_1$, $S\colon F_1\to F$, and a continuous function $g\colon E_1\to F_1$ such that for every $x\in K$,
\begin{equation}
  f(x) - S\circ g\circ T(x)\in W.
  \label{eq:aim}
\end{equation}
Without loss in generality, we can assume $W$ to be circled, that is, $\lambda W\subseteq W$ when $\abs\lambda\leq 1$. And we will prove the above with $2W$ instead of $W$, which of course makes no difference.

We start again by constructing $F_1$. Since the restriction of $f$ to the compact set $K$ is uniformly continuous, there is a neighbourhood of identity $V$ in $E$ with the property
\[\forall x,y\in K,~~x-y\in V\Rightarrow f(x)-f(y)\in W.\]
Select a finite subcover of the open cover of $K$ with the open sets $x+V$, $x\in K$. In other words, choose $x_1,x_2,\ldots,x_n\in K$ so that $\{x_1,\ldots,x_n\}+V\supseteq K$. Using Theorem \ref{th:unity}, select a continuous partition of unity on $K$ subordinated to the open cover with sets $x_i+V$, $i=1,2,\ldots, n$. That is, for each $i$, choose a continuous function $h_i\colon K\to [0,1]$ vanishing on $K$ outside of $x_i+V$ and such that $\sum_{i=1}^n h_i(x)=1$ for all $x\in K$. Define the function $\tilde f\colon K\to F$ by
\[\tilde f(x) = \sum_{i=1}^n h_i(x)f(x_i).\]
This function is continuous. It takes values in a finite-dimensional subspace of $F$, spanned by the vectors $f(x_i)$, $i=1,2,\ldots,n$, which subspace we denote $F_1$. Finally, for every $x\in K$,
\begin{align}
  f(x)-\tilde f(x) & =  \sum_{i=1}^nh_i(x)[f(x)-f(x_i)]\nonumber \\
  & = \sum_{x-x_i\in V}h_i(x)[f(x)-f(x_i)]\nonumber \\
  &\in \left( \sum_{x-x_i\in V}h_i(x)\right) W\nonumber \\
  &\subseteq W. \label{eq:tildefp}
\end{align}

\begin{remark}
  The above is a standard construction in the theory of tensor products of locally convex spaces, see, e.g., the proof of Proposition 44.2, part 1, in \cite{treves2}.
\end{remark}

Again, the operator $S\colon F_1\to F$ will be nothing more than the canonical subspace embedding $F_1\hookrightarrow F$.

Denote $\lin K$ the locally convex subspace of $E$ algebraically spanned by $K$. Suppose we manage to finish the construction with $\lin K$ in place of $E$, that is, we have a finite rank continuous linear operator $T\colon \lin K\to E_1\cong\R^m$, $m=\dim E_1$, and a continuous function $g\colon E_1\to F_1$ satisfying Eq. (\ref{eq:aim}). Applying the Hahn--Banach theorem (Corol. \ref{c:hb}),
we extend $T$ from $\lin K$ to a continuous linear operator $E\to E_1$. Obviously, for the extended operator, the property in Eq. (\ref{eq:aim}) still holds. Thus, from now on, we will assume without loss in generality that $E$ is spanned by the compact subset $K$.

Fix any norm on the finite-dimensional space $F_1$.
Since the restriction of $\tilde f$ to $K$ is uniformly continuous, for every $k\in\N_+$ there is a continuous prenorm $p_k$ on $E$ so that,
\begin{equation}
  \forall x,y\in K,~~p_k(x-y)<1 \Rightarrow \norm{\tilde f(x)-\tilde f(y)}_{F_1}<\frac 1k. \label{eq:1/k}
\end{equation}
Let $\pi_k\colon E\to E/p_k$ be the canonical quotient map of $E$ into the normed space $E/p_k$.
The direct product of the normed quotients $E_k$, $k=1,2,\ldots$ equipped with the product topology is a
locally convex space.
Define a mapping $\varpi\colon E\to\prod_{k\in\N_+}E/p_k$ by
\[\varpi(x)=\left(\pi_k(x)\right)_{k=1}^{\infty}.\]
This is a continuous linear operator. If $\varpi(x)=\varpi(y)$ for $x,y\in K$, then the property in eq. (\ref{eq:1/k}) holds for all $k$, so $\tilde f(x)=\tilde f(y)$. This means the mapping $\tilde f$ factors through $\varpi$: there is a mapping $\bar f\colon \varpi(K)\to F_1$ with
\begin{equation}
  \bar f\circ \varpi=\tilde f.
  \label{eq:factors}
  \end{equation}

Note that $\tilde f$ and $\varpi\vert_K$ are both continuous, and since $K$ is compact, they are closed mappings (Prop. \ref{p:closed}). It follows from Prop. \ref{p:3maps} that $\bar f$ is continuous as well.

Since we are now assuming $E$ to be spanned by $K$, each space $E/p_k$ is spanned by 
the image $\pi_k(K)$. It is therefore separable because $\pi_k(K)$ is a compact metric space.
Choose for every $k$ a linear injection of norm $1$ into a separable Hilbert space, $i_k\colon E/p_k\to \ell^2$ (Corol. \ref{c:iell2}). This defines a continuous injective linear operator
\[i\colon \prod_{i=1}^{\infty} E/p_k\ni (x_k)_{k=1}^{\infty}\mapsto (i_k(x_k))_{k=1}^{\infty}\in \ell_2^{\N}.\]
Here $\ell_2^{\N}$ stands for the product of countably many copies of the separable Hilbert space $\ell_2$, equipped with the product topology.
It is a multi-Hilbert Fr\'echet (complete metrizable) locally convex space. The restriction of $i$ on $\varpi(K)$ is continuous and injective, hence a homeomorphism by Prop. \ref{p:homeo}, so the function
\begin{equation}
  \breve f = \bar f\circ (i\vert_{\varpi(K)})^{-1}\colon i\circ \varpi(K)\to F_1
  \label{eq:breve}
  \end{equation}
is continuous. Hence, it is uniformly continuous, and extends to a uniformly continuous function from $\ell_2^{\N}$ to $F_1$ by Corol. \ref{c:katetov}. 
We will retain the same symbol $\breve f$ for this extension.

Select a neighbourhood of zero, $O$, in $\ell_2^{\N}$ with the property
\begin{equation}
  \forall x,y\in \ell_2^{\N},~x-y\in O\Rightarrow \breve f(x)-\breve f(y)\in W.
  \label{eq:brefefF}
\end{equation}

By Corol. \ref{c:multihilbertAP}, the space $\ell_2^{\N}$ has the approximation property. Choose a continuous linear operator $\pi$ from $\ell_2^{\N}$ to a finite-dimensional subspace, $E_1=\pi\left(\ell_2^{\N}\right)$, so that
\[\forall x\in i\circ\varpi (K),~~\pi(x)-x\in O.\]
By Eq. (\ref{eq:brefefF}), we have for such $x$
\begin{equation}
\breve f(x)-\breve f(\pi(x))\in W.
\label{eq:breveeps}
\end{equation}

The composition $\pi\circ i\circ\varpi$ is a bounded linear operator from $E$ to the finite-dimensional space $E_1$.
Finally, define a continuous function $g\colon E_1\to F_1$ by $g=\breve f\vert_{E_1}$.

Now let $x\in K$. Denote $z= i\circ\varpi(x)\in i\circ\varpi(K)$.
We have

\begin{align*}
  \mbox{\tiny eq. (\ref{eq:factors}) }~~  \tilde f(x) & = \bar f\circ \varpi(x) \\
  &= \bar f\circ (i\vert_{\varpi(K)})^{-1}\circ i\circ\varpi(x) \\
\mbox{\tiny eq. (\ref{eq:breve})}  & = \breve f\circ i\circ\varpi(x) \\
  &= \breve f(z),
  \end{align*}
hence
\begin{align*}
  f(x)-S\circ g\circ T(x) &= f(x) - \breve f(\pi(z)) \\
& =
  (f(x) -\tilde f(x)) + (\tilde f(x)- \breve f(\pi(z)) \\
  &=  (f(x) -\tilde f(x)) + (\breve f(z)- \breve f(\pi(z)) \\
\mbox{\tiny eqs. (\ref{eq:tildefp}), (\ref{eq:breveeps})}  &\in W + W \\
  &= 2W.
  \end{align*}

\section{Smooth decoding/encoding\label{s:smooth}}

In this Section, we prove Theorem \ref{th:main2}.

\subsection{(1)$\Rightarrow$(2): Necessity of the Approximation Property}
Consider the identity mapping $f=\Id_E\colon E\to E$, $f(x)=x$. According to Example \ref{ex:linear}, at every point $x\in E$ 
\begin{equation}
  D^1f(x) = \Id_E.
  \label{eq:identity}
  \end{equation}

Now let $K$ be a compact subset of $E$, and let $V$ a neighbourhood of zero in $E$.
The subset $\{0\}\times K$ is compact in $E\times E$. 
By our assumption on $E$, there exist finite-dimensional spaces $E_1,F_1$ and continuous linear operators $T\colon E\to E_1$, $S\colon F_1\to E$ such that for every $y\in K$,
\begin{equation}
  D^1f(0)(y) - D^1(SgT)(0)(y)\in V.
  \label{eq:approxim}
\end{equation}
By the Chain Rule \ref{th:chain} and Example \ref{ex:linear},
\begin{align*}
  D^1(SgT)(0) &= S\circ D^1g(T(0))\circ T\\
  &= S\circ D^1g(0)\circ T.
\end{align*}
Notice that $R=S\circ D^1g(0)\circ T$ is a finite rank continuous linear operator.
Eqs. (\ref{eq:identity},\ref{eq:approxim}) imply:
\[\forall y\in K, ~y-R(y)\in V.\]
In other words, the identity operator $\Id_E$ is approximated in the compact-open topology with finite rank operators. This means that $E$ has the AP.

\subsection{(2)$\Rightarrow$(1): Sufficiency of the Approximation Property}
Let $f\colon E\to F$ be a $C^k$-function, where $F$ is a locally convex space. Let $K\subseteq E$ be a compact subset, and $V$ a neighbourhood of zero in $F$.

The derivatives $D^if$ are all continuous as mappings between locally convex spaces $E\times E^i$ and $F$, so one can apply lemma \ref{l:uniformnbhd} and choose a neighbourhood $U$ of zero in $E$ with the property that for all $i=0,1,\ldots,k$, 
$x,v_1,\ldots,v_i\in K$ and $u_0,u_1,\ldots,u_k\in U$, 
\[D^if(x)(v_1,\ldots,v_i)-D^if(x+u_0)(v_1+u_1,\ldots,v_i+u_i)\in V.\]
Now use the assumed approximation property of $E$ to find a finite rank continuous linear operator $\pi\colon E\to E$ so that for each $x\in K$, $\pi(x)-x\in U$. For each $i$ as above, we have by Lemma \ref{l:compo}
\[D^i(f\circ\pi)(x)(v_1,\ldots,v_i)= D^i(f(\pi(x))(\pi(v_1),\ldots,\pi(v_i)),
\]
from where
\[D^i(f\circ\pi)(x)(v_1,\ldots,v_i)-D^if(x)(v_1,\ldots,v_i)\in V.\]
Denote $E_1=\pi(E)$ and $f_1=f\vert_{E_1}\colon E_1\to F$. And $\pi$ will serve as the operator $T$.

According to Theorem \ref{th:treves}, for some $g\in C^k(E_1,\R)\otimes F\subseteq C^k(E_1,F)$ 
\begin{equation}
  D^if_1(x)(v_1,\ldots,v_i)-D^ig(x)(v_1,\ldots,v_i)\in W
  \label{eq:g}
\end{equation}
for all $i=0,1,\ldots,k$ and $x,v_1,\ldots,v_i\in \pi(K)$. This $g$ is a $C^k$ function from $E_1$ to some finite-dimensional subspace $F_1$ of $F$. Again applying Lemma \ref{l:compo}, we conclude from eq. (\ref{eq:g}): for all $i$ and $x,v_1,\ldots,v_i\in K$,
\begin{equation}
  D^i(f_1\circ\pi)(x)(v_1,\ldots,v_i)-
    D^i(g\circ\pi)(x)(v_1,\ldots,v_i)\in W.
  \end{equation}
Take as $S$ the canonical embedding $F_1\hookrightarrow F$.
We conclude: for each $i=0,1,\ldots,k$ and every $x,v_1,\ldots,v_k\in K$,

\begin{align*}
   D^if(x)(v_1,\ldots,v_i)& -  D^i(S\circ g\circ T)(x)(v_1,\ldots,v_i) = \\ &
  D^if(x)(v_1,\ldots,v_i)- 
  D^i(f\circ\pi)(x)(v_1,\ldots,v_i) + \\ & D^i(f\circ\pi)(x)(v_1,\ldots,v_i)-D^i(g\circ \pi)(x)(v_1,\ldots,v_i)\\
 & \in W+W = 2W.
\end{align*}


\begin{thebibliography}{00}

\bibitem{banach} Stefan Banach, {\em Th\'eorie des op\'erations lin\'eaires,}
  Warszawa, 1932.

\bibitem{bastiani}
  A. Bastiani, {\em Applications diff\'erentiables et vari\'et\'es diff\'erentiables de dimension infinie,} J. Anal. Math. \textbf{13} (1964), 1--114.

\bibitem{BHKS}
  Kaushik Bhattacharya, Bamdad Hosseini, Nikola B Kovachki, and Andrew M Stuart, {\em Model reduction and neural networks for parametric PDEs,} SMAI Journal of Computational Mathematics \textbf{7} (2021), 121--157.

\bibitem{engelking}
R. Engelking, \textit{General Topology}, 
Math. Monographs, {\bf 60}, PWN - Polish Scient. Publishers, Warsaw, 1977.

\bibitem{HMVZ}
P. H\'ajek, V. Montesinos Santaluc\'\i a, J. Vanderwerff, V. Zizler,
{\em Biorthogonal systems in Banach spaces,}
CMS Books Math./Ouvrages Math. SMC, \textbf{26},
Springer, New York, 2008.

\bibitem{jarchow}
  Hans Jarchow, {\em Locally convex spaces,} B.G. Teubner, Stuttgart, 1981.

\bibitem{katetov}
M. Kat\v etov, {\em On real-valued functions in topological spaces,}
Fund. Math. \textbf{38} (1951), 85--91; {\em corrigendum,} Fund. Math. \textbf{40} (1953), 203--205.

\bibitem{keller} Heinrich Hans Keller, {\em Differential Calculus in Locally Convex Spaces,} Springer Berlin Heidelberg, 1974.

\bibitem{KG}  A.A. Kirillov and A.D. Gvishiani, {\em Theorems and Problems in Functional Analysis,} translated from Russian by Harold H. McFaden, Springer-Verlag, New York, 1982.

\bibitem{KLS}
Nikola B. Kovachki, Samuel Lanthaler and Andrew M. Stuart, {\em Operator learning: Algorithms and analysis,} Handbook of Numerical Analysis \textbf{25} (2024), 419--467

\bibitem{KLLABSA}
  Nikola Kovachki, Zongyi Li, Burigede Liu, Kamyar Azizzadenesheli, Kaushik Bhattacharya, Andrew Stuart, Anima Anandkuma, {\em  Neural Operator: Learning Maps Between Function Spaces With Applications to PDEs,} Journal of Machine Learning Research \textbf{24} (2023), 89:1--97, 2023.

\bibitem{LJK}
  Lu Lu, Pengzhan Jin, and George Em Karniadakis, {\em DeepONet: Learning nonlinear operators for identifying differential equations based on the universal approximation theorem of operators,} Nature Machine Intelligence \textbf{3} (2021), 218--229.

\bibitem{michael} E. Michael, {\em Continuous selections. I,} Annals of Math. \textbf{63} (1956), 361--382.

\bibitem{michal}
  A.D. Michal, {\em Differential calculus in linear topological spaces,} Proc. Nat. Acad. Sci. U.S.A. \textbf{24} (1938), 340--342.

\bibitem{neeb}
 Karl-Hermann Neeb, {\em Towards a Lie theory of locally convex groups,}
Jpn. J. Math. \textbf{1} (2006), 291--468.

\bibitem{schaefer}
  Helmut H. Schaefer, {\em Topological Vector Spaces,} Springer, New York, 1971.

\bibitem{treves} François Tr\`eves, {\em Locally convex spaces and linear partial differential equations,} Springer--Verlag, New York, 1967.

\bibitem{treves2} François Tr\`eves, {\em 
Topological vector spaces, distributions and kernels,}
Unabridged republication of the 1967 original,
Dover Publications, Inc., Mineola, NY, 2006.

\bibitem{weil}
  Andr\'e Weil, {\em Sur les espaces \`a structure uniforme et sur la topologie g\'en\'erale,} Act. Sci. Ind. \textbf{551}, Paris, 1937.
\end{thebibliography}
\end{document}